\documentclass{amsart}
\usepackage{verbatim, graphics, color}
\usepackage{tikz,pgf}
\usepackage{mathrsfs}
\usepackage{amsmath,amsthm,amsfonts}

\newtheorem{theorem}{Theorem}
\newtheorem{proposition}[theorem]{Proposition}

\newtheorem{definition}{Definition}

\author{Jacob A. White}
\address{Mathematical Sciences Research Institute, 17 Gauss Way, Berkeley, CA}
\email{jawhite@msri.org}
\thanks{
The author was partially supported by an NSF grant DMS-0932078,
administered by the Mathematical Sciences Research Institute while the
author was in residence at MSRI during the Complementary Program, Fall 2010.
This work began during the
visit of the author to MSRI and we thank the institute for its
hospitality.
}
\subjclass[2000]{Primary 05C31 Secondary 05C15 Secondary 05C65}

\date{}

\title[On Multivariate Chromatic Polynomials of Hypergraphs]{On Multivariate Chromatic Polynomials of Hypergraphs and Hyperedge Elimination}
\keywords{Hypergraphs, Colorings, Tutte Polynomials}

\begin{document}

\maketitle

\begin{abstract}

In this paper, we consider multivariate hyperedge elimination polynomials and multivariate chromatic polynomials for hypergraphs. The first set of polynomials is defined in 
terms of a deletion-contraction-extraction recurrence, previously investigated for graphs by Averbouch, Godlin, and Makowsky. The multivariate chromatic polynomial is an equivalent 
polynomial defined in terms of colorings, and generalizes the coboundary polynomial of Crapo, and the bivariate chromatic polynomial of Dohmen, P\"onitz and Tittman. We show that specializations of 
these new polynomials recover
 polynomials which enumerate hyperedge coverings, matchings, transversals, and section hypergraphs. 
We also prove that the polynomials can be defined in terms of M\"obius inversion on the bond lattice of a hypergraph, as well as compute these polynomials for various classes of hypergraphs.
\end{abstract}

\section{Introduction}

The chromatic polynomial of a graph enumerates the number of proper $k$-colorings of a graph. It was originally introduced by Birkhoff, who hoped that understanding this polynomial could lead to 
a proof of the four-color theorem. The chromatic polynomial was generalized by Tutte to give a two-variable polynomial, now called 
the Tutte polynomial (\cite{tutte-graph-theory}, \cite{tutte-ring}). The Tutte polynomial is actually defined for matroids in general, and have many wonderful enumerative properties. A good survey 
of these polynomials appears in \emph{Matroids and their Applications} \cite{brylawski-oxley}. 

There has been a recent resurgence in the study of graph polynomials. Prominent examples include the interlace polynomials, matching polynomials, independent set polynomials, and the edge elimination polynomial. 
This last polynomial is the main focus of this paper.
It was introduced by Averbouch, Godlin, and Makowsky \cite{averbouch-godlin-makowsky}. It is defined recursively in terms of three graph-theoretic operations - 
deletion, contraction, and extraction. Extraction does not appear to be a 
matroidal operation. In fact, not all trees on $n$ vertices have the same ege elimination polynomial, despite having isomorphic cycle matroids.

There are two major purposes of this present paper. The first is to introduce a new polynomial, the multivariate chromatic polynomial. This polynomial is equivalent to the edge elimination polynomial, 
as both can be obtained from each other by substitution. Moreover, we prove several results for the multivariate chromatic polynomial, which generalize known results about the coboundary polynomial, and the 
bivariate chromatic polynomial of Dohmen, P\"onitz, and Tittman \cite{dohmen-ponitz-tittman}. The second purpose of the paper is to present results for hypergraphs. 
Since it makes sense to define deletion, contraction, and extraction for hypergraphs, it also makes sense to prove results in this level of generality.

We show that certain evaluations of the hyperedge elimination polynomial or multivariate chromatic polynomial give polynomials that enumerate matchings, stable sets, hyperedge coverings, 
and section hypergraphs. The first two results are extensions of known results, but the latter two are new. Also, the Tutte polynomial is a specialization of these polynomials. We also give a 
subset expansion formula for the hyperedge elimination polynomial, generalizing the work of Averbouch et al, as well as a M\"obius inversion formula, generalizing known results for the coboundary 
polynomial and the bivariate chromatic polynomial. Furthermore, we are able to give formulas for the multivariate chromatic polynomial for particular classes of hypergraphs.

\section{Review of Hypergraph Terminology}

A hypergraph is a pair $(V, E)$, where $V$ is a finite set of vertices, and $E = \{e_i: i \in I, e_i \subset V \}$ is a collection of hyperedges (the hyperedges are
 indexed by some set $I$). Note that this means that we can have $i \neq j \in I$ with $e_i = e_j$. That is, we are allowing multiple edges in our hypergraphs. A $k$-edge is an edge with $k$ 
vertices. We are allowing $1$-edges, which in the context of coloring is equivalent to the notion of loops in a graph. We are also allowing empty edges, although we will only concern ourselves with 
empty edges when studying duality in hypergraphs. Finally, a pair of edges $e_i, e_j$ with $i \neq j$ are parallel if
 $e_i = e_j$. 

Given a hypergraph $H$, there are two notions of induced subgraph. Given a subset $A$ of vertices, a \emph{subhypergraph} is the hypergraph 
$H_A = (A, \{e_i \cap A: e_i \cap A \neq \emptyset \})$. Note that of course that the new index set for edges is $\{i \in I: e_i \cap A \neq \emptyset \}$. A \emph{vertex section hypergraph} 
is the hypergraph $H \times A = (A, \{e_i: e_i \subset A)$. 

Given a subset $J \subset I$, let $E_J = \{e_j: j \in J \}$. The \emph{partial hypergraph} $H_J = (V, E_J)$, and the \emph{edge section} hypergraph $H \times J$ has edge set $E_J$ 
and vertex set $\cup_{i \in J} e_i$. Note that in context one should be able to see the difference between partial hypergraph and subhypergraph. For disjoint sets $J, K \subset I$, 
we refer to $(J, K)$ as a vertex disjoint pair if $e \cap f = \emptyset$ for all $e \in E_J, f \in E_K$.

Let $H$ be a hypergraph, and let $e_i$ be a hyperedge. The \emph{deletion} is the hypergraph $H - e_i = (V, \{e_j: j \neq i \})$. 
The \emph{extraction} $H \dagger e_i$ is the hypergraph $(V \setminus e_i, \{e_j: e_j \cap e_i = \emptyset \})$.
The \emph{contraction} $H/e$ is obtained from $H \dagger e_i$ by adding one new vertex $v_i$, and edges $\{e'_j: i \neq j, e_j \cap e \neq \emptyset, e'_j = (e_j \setminus e_i) \cup \{v_i \})$. 
We shall refer to edges $e'_j$ as the partially contracted edges. We are interested in studying hypergraph polynomials that can be defined recursively in terms of these three operations. 
The most general polynomial satisfying such a recurrence will be called the \emph{hyperedge elimination polynomial}. 
An example of deletion, contraction, and extraction is given in Figure \ref{elimination}

\begin{figure}[htbp]
 \begin{center}
  \includegraphics[height=6cm]{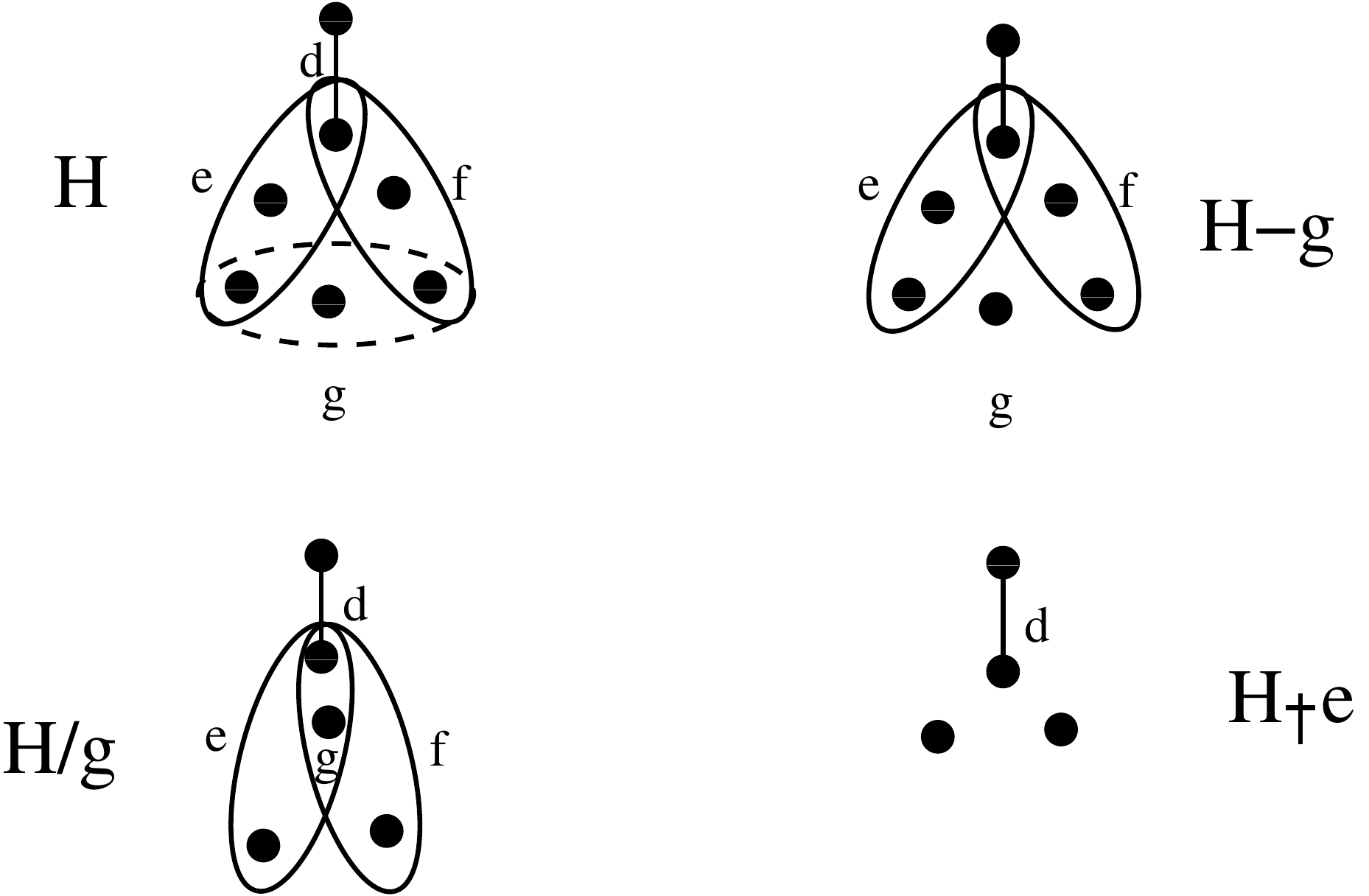}
 \end{center}
\caption{An example of deletion, contraction, and extraction}
\label{elimination}
\end{figure}

A \emph{chain} is a sequence $v_0, e_1, v_1, \ldots, e_k, v_k$, where $v_i \in e_i$ for $1 \leq i \leq k$, $v_i \in e_{i+1}$ for $0 \leq i \leq k-1$, and $e_1, \ldots, e_k$ are edges. 
If the edges are all distinct, we obtain a \emph{path}. If $k > 2$ and $v_0 = v_k$, we call the path a \emph{cycl}e. We say that a hypergraph is connected if for every two vertices $u$ 
and $v$ there exists a path with $v_0 = u$, $v_k = v$. As with graphs, a hypergraph decomposes into connected components. Let $k(H)$ denote the number of connected components of a hypergraph.

The decision to study multivariate polynomials is motivated by the work of Sokal \cite{sokal}. Relationships between graph polynomial sometimes have a multivariate analogue that is easier to prove.
For hypergraphs, there are at least two more reasons to focus on studying multivariate polynomials, with indeterminates for each hyperedge. First, given a set of indeterminates 
$w_0, w_1, \ldots$, we can make the substitution $t_e = t_{|e|}$, and obtain new polynomials. Secondly, we can do a further substitution, replacing $t_i$ with $t^i$. These resulting 
polynomials cannot be obtained from the edge elimination polynomials for general hypergraphs, yet they contain some very refined data regarding the structure of a graph.

Let $E_0$ denote the hypergraph with no vertices or edges. Let $E_1$ denote the hypergraph with only one vertex, and no edges.

\section{A List of Interesting Polynomials}
In this section, we define several polynomials involving combinatorial aspects of hypergraphs that are often studied. Most of these polynomials generalize to well-known graph polynomials, 
but a few of these polynomials are actually new. Throughout, fix a hypergraph $H$ with vertex set $V$ and edge set $E$. Let $n$ be the number of vertices, $m$ be the number of edges. 
Given any set $F \subseteq E$, let $m_*(F) = \sum_{e_i \in F} |e_i|$ (note that certain authors use $m_*(F)$ to denote this summation $ \sum_{e_i \in F} (|e_i| - 1)$).

A \emph{matching} in $H$ is a set $F \subseteq E$ of edges such that $e_i \cap e_j = \emptyset$ for all $e_i \neq e_j \in F$, $i \neq j$. The \emph{bivariate matching polynomial} is 
defined by $\mu(H; x, y) = \sum_{M} x^{n - m_*(M)} y^{|M|}$, where the summation is over all matchings of $H$. This generalizes the bivariate matching polynomial studied by Iverbouch et al. 
The mulivariate matching polynomial is given by $\mu(H; x, \textbf{y}) = \sum_M x^{n - m_*(M)} \prod_{e_i \in M} y_{e_i}$. Note that this generalizes the multivariate matching polynomial 
studied by Averbouch and Makowsky. For graphs, the substitutions $x = 1$ and $y_{uv} = y_{uv} x_u x_v$ results in the multivariate polynomial originally introduced by Heilman and Lieb \cite{heilman-lieb}. So this is one 
case where we see that it is natural to consider the multivariate version of a polynomial.

A \emph{hyperedge covering} is a collection of edges $F \subset E$ such that $\cup_{e \in F} e = V$. That is, every vertex of $H$ lies on some hyperedge in the covering. 
A vertex is \emph{exposed} if it is contained in no hyperedges. The hyperedge covering polynomial is defined by $\kappa(H;x, y) = \sum_{C} x^{|C|} y^{k(H|_C)}$ where the sum is over all
hyperedge coverings $C$, and the polynomial is $0$ if $H$ has an exposed vertex. The multivariate hyperedge covering polynomial is defined by 
$\kappa(H; x, y, \textbf{t}) = \sum_C x^{|C|} y^{k(H|_C)} \prod_{e \in C} t_e$. This polynomial does not appear in the literature.

A transversal is a set $S \subseteq V$ such that $S \cap e \neq \emptyset$ for all $e \in E$. The \emph{transversal polynomial} is defined by $\tau(H; x) = \sum_{S} x^{|S|}$ 
where the sum is over all transversals $S$. When $H$ is a graph, this polynomial is known as the vertex-cover polynomial. Recall that the complement of a transversal is an independent set, 
so the vertex-cover polynomial is actually the same polynomial as the independent set polynomial also appearing in the literature.

The section polynomial of a hypergraph is defined by $S(H; x, y) = \sum_{S \subseteq V} x^{n - |S|} y^{m(H \times S)}$.
The multivariate section polynomial is defined by $S(H; x, \textbf{y}) = \sum_{S \subseteq V} x^{n - |S|} \prod_{e \in E(H \times S)} y_e$. For graphs, these polynomials enumerated induced 
subgraphs, and do not appear to have been studied. Note that setting $y = 0$ recovers the transversal polynomial. Also consider the substitution $y_e = y_{|e|}$. In the resulting polynomial, the coefficient 
of $x^i y_1^{i_1} \cdots y_m^{i_m}$ is the number of section hypergraphs with exactly $n-i$ vertices and $i_j$ edges of size $j$ for all $j$. Also note that the resulting polynomial is a hypergraph 
invariant.

The Tutte polynomial and multivariate Tutte polynomial of a hypergraph have previously been defined. Using the Potts model form of the definition, we define the multivariate Tutte 
polynomial by $Z(H; x, \textbf{t}) = \sum_{J \subseteq I} x^{k(H|_I)} \prod_{j \in J} t_{e_j}$.

There are two more polynomials we will define in this paper: the hyperedge elimination polynomial, and the multivariate chromatic polynomial. These two polynomials differ only by
 substitutions, and hence are equivalent polynomials. Moreover, all the polynomials of this section are all substitutions, up to prefactors, of the hyperedge elimination polynomial 
(and the multivariate chromatic polynomial). The table shows these polynomials, and the corresponding substitutions involved to obtain them.
\begin{table}[htbp]
\begin{center}
\begin{tabular}{|c|c||c|c|}
\hline
Polynomial & Substitution & Multivariate & Substitution \\
\hline
$\kappa(H; x,y)$ & $\xi(H; 0, x, xy)$ & $\kappa(H; x, y, \textbf{t})$ & $\xi(H; 0, x, xy, \textbf{t})$ \\
\hline
$\mu(H; x, y)$ & $\xi(H; x, 0, y)$ & $\mu(H; x,y, \textbf{t})$ & $\xi(H; x,0,z,\textbf{t})$ \\
\hline
$Z(H; x,y)$ & $\xi(H; x,y,0)$ & $Z(H; x, \textbf{t})$ & $\xi(H;x,1,0,\textbf{t})$ \\
 \hline
$P(H; p,q,t)$ & $P(H; q, t-1, (q-p)(t-1))$ & $P(H; p, q, \textbf{t})$ & $\xi(H; q, 1, p-q, \textbf{t}-1)$ \\
\hline
$S(H; x,y)$ & $P(H; 1, x+1, y)$ & $S(H;x,\textbf{y})$ & $P(H; 1, x+1, \textbf{y})$ \\
\hline
$P_c(H; p, t)$ & $P(H; p, p, t)$ & $P_c(H; p, \textbf{t})$ & $P(H; p, p, \textbf{t})$ \\
\hline
$\xi(H;x,y,z)$ & $P(H; x+\frac{z}{y}, x, y+1)$ & $\xi(H;x,y,z,\textbf{t})$ & $P(H; x+\frac{z}{y}, x, y\textbf{t}+1)$ \\
\hline
 
\end{tabular}
\end{center}
\caption{Some polynomials, and the corresponding evaluations of $\xi$ and $P$}
\end{table}

We have chosen to let $P_c$ represent the coboundary polynomial of a hypergraph (to avoid confusion with the chromatic polynomial).

\section{The Hyperedge Elimination Polynomial}

In this section, we define the hyperedge elimination polynomial. 
\begin{definition}
 Let $\xi(H; x,y,z)$ be defined by 
$$\xi(H; x,y,z) = \sum_{(I, J)} x^{k(H_{I \sqcup J}) - k(H \times J)} y^{|I|+|J| - k(H \times J)} z^{k(H \times J)}$$
where the sum is over vertex disjoint pairs $(I, J)$. 
\end{definition}

\begin{theorem}
 $\xi(H; x, y, z)$ satisfies the following:
\begin{enumerate}
 \item $\xi(E_0; x,y,z) = 1$
\item $\xi(E_1; x,y,z) = x$
\item $\xi(H_1 \sqcup H_2; x,y,z) = \xi(H_1; x,y,z) \cdot \xi(H_2; x,y,z)$
\item for any $e \in E(H)$, we have $\\ \xi(H;x,y,z) = \xi(H-e;x,y,z) + y\xi(H/e; x,y,z) + z\xi(H\dagger e; x,y,z)$
\end{enumerate}

\end{theorem}
We actually prove a similar recurrence of the multivariate hyperedge elimination polynomial in a later section. The above theorem follows through specialization.

\begin{proposition}
 Let $f$ be a function from hypergraphs to some integral domain $R$, such that $f$ is invariant under graph isomorphism, and $f$ satisfies a recurrence with parameters $\alpha, \beta, \gamma, \delta \in R$ subject to:
\begin{enumerate}
 \item $f(\emptyset) = 1$
\item $f((\{v\}, \emptyset)) = \alpha$
\item $f(H_1 \sqcup H_2) = f(H_1) \cdot f(H_2)$
\item for any $e \in E(H)$, we have $\\ f(H) = \beta f(H-e) + \gamma f(H/e) + \delta f(H\dagger e)$
\end{enumerate}
Then either:
\begin{enumerate}
 \item $\delta = 0$ and $f(H) = \beta^{m(H)} \xi(H; \alpha, \frac{\gamma}{\beta}, 0)$
\item $\beta = 1$ and $f(H) = \xi(H; \alpha, \gamma, \delta)$ 
\item $f(H) = \alpha^{n(H)} = \xi(H; 1, \alpha, 0)$
\end{enumerate}

\end{proposition}
\begin{proof}
 If $\delta = 0 $ or $\beta = 1$ we see that the corresponding evaluation of $\xi$ yields the same recursion as $f$, and hence the equality holds.

So assume $\delta \neq 0$ and $\beta \neq 1$. Let $H$ be a hypergraph, let $v$ be a vertex of $H$. Consider two new vertices $y, z$ that are not vertices of $H$, 
and construct a new hypergraph $G$, by adding vertices $y, z$ to $H$, and hyperedges $e$ with vertex set $vy$ and $f$ with vertex set $yz$. 

First, eliminate edge $e$, then eliminate edge $f$:
\begin{displaymath}
\begin{array}{ccl}
 
f(G) &  = & \beta f(G-e) + \gamma f(G/e) + \delta f(G\dagger e) \\ & = &  \beta (\beta f(G-e-f) + \gamma f(G-e/f) + \delta f(G-e \dagger f)) \\ & + &  
\gamma (\beta f(G/e-f) + \gamma f(G/e/f) + \delta f(G/e \dagger f)) + f(G \dagger e) \\ & = & (\beta^2 \alpha^2 + 2\alpha \beta \gamma + \beta \delta)f(H) + (\alpha \delta + \beta \delta) f(H-v)
\end{array}
\end{displaymath}

Instead, first eliminate edge $e$, then edge $f$ to obtain:
$$f(G) = (\beta^2 \alpha^2 + 2\alpha \beta \gamma + \delta)f(H) + (\alpha \beta \delta + \beta \delta) f(H-v)$$

Thus we obtain:
$$\alpha \delta f(H-v) + \beta \delta f(H) = \delta f(H) \alpha \beta \delta f(H-v)$$
or equivalently:
$$(1-\beta)\delta \alpha f(H-v) = (1 - \beta) \delta \alpha f(H)$$
Since $\delta \neq 0$, $\beta \neq 1$, and $R$ is an integral domain, we have $\alpha f(H - v) = f(H)$ for all hypergraphs $H$ and vertices $v$. Hence an inductive argument on $n$ yields that $f(H) = \alpha^{n(H)}$.
\end{proof}
Thus, any sort of function that obeys the hyperedge elimination recursion and is a graph invariant must be, up to prefactor, an evaluation of the hyperedge elimination polynomial. 
Note that this argument and result for graphs was previously given by Averbouch et al \cite{averbouch-godlin-makowsky}. They also choose to show that a polynomial defined in terms of 
this recurrence gives a graph invariant in the case $w = 1$. However, the fact that $\xi$, defined as a subset expansion, satisfies this recurrence already proves this fact, so we will 
not give the alternate, algebraic proof.

\section{The Multivariate Hyperedge Elimination Polynomial}
Now we define the multivariate hyperedge elimination polynomial. This is a hypergraph extension of the labeled edge elimination polynomial, and we denote it $\xi(G; x,y,z,\textbf{t})$. 

\begin{definition}
 Let $\xi(H; x,y,z, \textbf{t})$ be defined by 
$$\xi(H; x,y,z, \textbf{t}) = \sum_{(I, J)} x^{k(H_{I \sqcup J}) - k(H \times J)} y^{|I|+|J| - k(H \times J)} z^{k(H \times J)} \prod_{i \in I \sqcup J} t_{e_i}$$
where the sum is over vertex disjoint pairs $(I, J)$. 
\end{definition}

\begin{theorem}
 $\xi(H; x, y, z)$ satisfies the following:
\begin{enumerate}
 \item $\xi(E_0; x,y,z,\textbf{t}) = 1$
\item $\xi(E_1; x,y,z, \textbf{t}) = x$
\item $\xi(H_1 \sqcup H_2; x,y,z,\textbf{t}) = \xi(H_1; x,y,z,\textbf{t}) \cdot \xi(H_2; x,y,z, \textbf{t})$
\item for any $e \in E(H)$, we have $\\ \begin{array}{ccl} \\ \xi(H; x,y,z,\textbf{t}) & = & \xi(H-e; x,y,z,\textbf{t}_{\neq e}) \\ 
& + & yt_{e}\xi(H/e;x,y,z, \textbf{t}_{\neq e}) \\ & + & z t_{e}\xi(H\dagger e; x,y,z,\textbf{t}_{\bot e})\end{array} \\$ where $\textbf{t}_{\neq e} = \{t_f: f \in E(H - e) \}$ and $\textbf{t}_{\bot e} = \{t_f: f \in E(H \dagger e) \}$. 
\end{enumerate}

\end{theorem}
\begin{proof}
 Let $(I, J)$ be a vertex disjoint pair. Let $p(H; I, J) = \\ x^{k(H_{I \sqcup J}) - k(H \times J)} y^{|I|+|J| - k(H \times J)} z^{k(H \times J)} \prod_{i \in I \sqcup J}$. Observe that $e \not\in E_{I \sqcup J}$ if and only if $(I,J)$ is a vertex disjoint pair for $H - e$. In such a case we see that $p(H -e; I, J) = p(H; I, J)$.
Now suppose $e \in E_J$, say $e_j$, and that it is in its own component. Then $(I, J - j)$ is a vertex disjoint pair for $H \dagger e$, and moreover $p(H; I, J) = zt_{e} p(H \dagger e; I, J-j)$. 
Suppose $e_j$ is not in an isolated component. Then $(I, J-j)$ is a vertex disjoint pairt for $H/e$ and $p(H; I, J) = yt_{e} p(H/e; I, J-j)$. Moreover, this covers all vertex disjoint pairs $(I, J)$ of $H/e$ for which $(I, J+j)$ is vertex disjoint for $H$ but $(I+j, J)$ is not.
Now suppose $e = e_i, i \in I$. Then $(I-i, J)$ is vertex disjoint for $H/e$, and $p(H; I, J) = p(H/e; I-i, J)$. This covers all vertex disjoint pairs $(I, J)$ of $H/e$ for which $(I+i, J)$ is a vertex disjoint pair for $H$. 
Thus, summing over all $(I, J)$, we obtain the result.
\end{proof}

Similar to \cite{averbouch-godlin-makowsky}, we could study the most general hyperedge elimination recurrence. 
We would discover some necessary conditions for a polynomial satisfying such a recurrence to be independent of the order of edges chosen.
Likewise, we would find sufficient conditions, and every polynomial meeting these sufficient conditions is, up to prefactor, and evaluation of the 
multivariate hyperedge elimination polynomial. However, it appears that not all polynomials which satisfy the hyperedge elimination recurrence are evaluations of $\xi(H;x,y,z, \textbf{t})$.

\subsection{Substitutions of the Multivariate Hyperedge Elimination Polynomial}

In this section, we consider what happens when we evaluate some of the variables in $\xi(H; x,y,z, \textbf{t})$ at $0$. 
Most of these evaluations extend known results for graphs. Note that $\xi(H; x,y,z, \textbf{t}) = x^n$ when $t_e = 0$ for all edges $e$.

First, if we evaluate at $z = 0$, we see that our hyperede elimination recurrence only involves deletion and contraction, and thus we end up with the 
multivariate Tutte polynomial (defined in terms of the Potts model). Note that for graphs, this notion of multivariate Tutte polynomial is equivalent to the usual definition of 
the Tutte polynomial, up to some substitution of variables and prefactor. That is, $\xi(H; x,y,0, \textbf{t}) = Z(H; x, y\textbf{t})$, and 
it is known \cite{sokal} $T(H; x,y) = (x-1)^{-c(H)} (y-1)^{-|V|} Z(H; (x-1)(y-1), y-1)$, where $T(H; x,y)$ is the usual Tutte polynomial.

Now we consider setting $y = 0$. In this case, our recurrence only involves deletion and extraction. One can check that the resulting recurrence is satisfied by the multivariate matching 
polynomial $\mu(H; x,y, \textbf{t})$. Consider an edge $e$. Any matching not involving $e$ is enumerated in $H-e$. Any matching involving $e$ is enumerated by $y_e \mu(H \dagger e; x, \textbf{y}_{\neq e})$.
So $\mu(H; x,\textbf{y}) = \mu(H; x, \textbf{y}) + y_e \mu(H; x, \textbf{y}_{\neq e})$. Hence $\xi(H; x, 0, z, \textbf{t}) = \mu(H; x, z\textbf{t})$.

Finally, consider the case $x = 0$. This case was not considered by Averbouch et al \cite{averbouch-godlin-makowsky}, 
and actually yields an interesting polynomial for hypergraphs. If $H$ has an isolated vertex, the result is $0$. So suppose $H$ has no isolated vertices. 
In terms of subset expansions, the only terms that do not vanish correspond to disjoint vertex pairs $(I, J)$ for which $k(H_{I \sqcup J}) = k(H \times J)$. 
In such a situation $I = \emptyset$, and $H$ must not have any isolated vertices. Then we see that $k(H_J) = k(H \times J)$, which is true if and only if $E_J$ is a hyperedge cover of $H$. 
Therefore the subset expansion reduces to a summation over hyperedge coverings, and we see that we obtain the hyperedge cover polynomial.
Thus $\kappa(H; x, y, \bar{t}) = \xi(H; 0, x, xy, \bar{t})$. One could also verify the hyperedge elimination recurrence, and the initial condition $\kappa(E_1, x,y, \bar{t}) = 0$.

\subsection{The Multivariate Hyperedge Elimination Polynomial for certain hypergraphs}

Let $P_{m, r}$ be the $r$-uniform elementary path hypergraph with $m$ edges, where $r \geq 2$, $m \geq 1$. That is, $P_{m,r}$ has edges $e_1, \ldots, e_m$, where each edge has exactly $r$ vertices. Moreover, $|e_{i-1} \cap e_i| = 1$ for $2 \leq i \leq m$, and $|e_{i} \cap e_{i+1}| = 1$ for $1 \leq i \leq m -1$, and these edges sets are otherwise disjoint.

Let $P_{m,r}(x,y,z) = \xi(P_{m,r}; x,y,z)$. For fixed $r$, we have thus defined a sequence of polynomials.

Consider applying hyperedge elimination to the edge $e_1$. We see that $P_{m,r} - e_1$ has $r-1$ isolated vertices, and then $P_{m-1,r}$ as the remaining component. Thus $\xi(P_{m,r} - e) = x^{r-1} \xi(P_{m-1}, r)$. Similarly, $\xi(P_{m,r}/e) = \xi(P_{m-1, r})$, and $\xi(P_{m,r} \dagger e) = x^{r-2} \xi(P_{m-2, r})$. 

Thus $P_{m,r}(x,y,z) = (x^{r-1} + y) P_{m-1, r}(x,y,z) + zx^{r-2} P_{m-2, r}(x,y,z)$, for $m > 2$, with initial conditions $P_{0,r}(x,y,z) = 1$ and $P_{1,r}(x,y,z) = x^r + yxt_e + zt_e$. Setting $x= z = 1$ and $y = x-1$, we specialize to the Fibonacci polynomials. That is $P_{m,r}(1,x-1, 1) = F_{m+2}(x)$. From this recurrence, one can show that the generating function $P_r(x,y,z,q) = \sum_{m \geq 0} P_{m,r}(x,y,z)q^m$ is given by $\frac{1+x^r+yx+z+q(x^{r-1}+y)}{1-q(x^{r-1}+y)-q^2zx^{r-2}}$.

Let $C_{m,r}$ be the $r$-uniform elementary hypercycle with $m \geq 3$ edges. That is, $C_{m,r}$ is obtained from $P_{m,r}$ by identifying some vertex in $e_1 \setminus e_2$ and $e_m \setminus e_{m-1}$.
Let $C_{m,r}(x,y,z) = \xi(C_{m,r}; x,y,z)$. Then $C_{m,r}(x,y,z) = x^{r-2}P_{m-1,r}(x,y,z) + yC_{m-1,r}(x,y,z) + zx^{2r-4}P_{m-3, r}(x,y,z)$. This is shown by applying hyperedge elimination to any 
edge of $C_{m,r}$. One can use this to give a recurrence for $C_{m,r}$ of order 3. However, we do not take this approach, as it is tedious.

\section{The Multivariate Chromatic Polynomial}

Now we define the multivariate chromatic polynomial, which generalizes the coboundary polynomial and bivariate chromatic polynomial of  graph. Let $q$ be a positive integer. 
Then a function $f: V(H) \to [q]$ is refered to as a $q$-coloring. A hyperedge $e_i$ is monochromatic if $f(u) = f(v)$ for all $u, v \in e$. The coloring $f$ is proper if it has no 
monochromatic hyperedge. Let $p \leq q$ be a positive integer. We now view $1, \ldots, p$ as \emph{primary} colors. A \emph{primary} hyperedge edge $e$ satisfying $f(u) = f(v) \leq p$ for all 
$u,v \in e$. Given a coloring $f$, let $P(f)$ denote the set of primary edges. That is, it is an edge whose vertices are all colored with the same primary color. 

Let $$P(H; p, q, \textbf{t}) = \sum_{f: V \to [q]} \prod_{e \in P(f)} t_{e}$$
We see that for every pair of positive integers $p \leq q$, this defines a multivariate polynomial in the $t_{e}$s. However, we will show that $P(H;p, q, \textbf{t})$ actually is a 
multivariate polynomial in variables $t_{e}$, $p$ and $q$. We call this polynomial the multivariate chromatic polynomial. Also, we call $P(H; p, \textbf{t}) = P(H; p, p, \textbf{t})$ 
the multivariate coboundary polynomial (as it specializes to the coboundary polynomial upon setting all $t_{e} = t$). We see that $P(H; p, q) = P(H;p, q, 0)$ is the bivariate chromatic polynomial 
defined by Dohmen P\"onitz and Tittman \cite{dohmen-ponitz-tittman}. Finally, let $P(H; p,q,t)$, the \emph{trivariate chromatic polynomial} be obtained by setting all $t_e = t$.
Next we prove a theorem that expresses the multivariate chromatic polynomial of $H$ in terms of coboundary polynomials of section hypergraphs of $H$. This generalizes a 
result of Dohmen et al \cite{dohmen-ponitz-tittman}.
\begin{theorem}
 $P(H;p,q,\textbf{t}) = \sum_{S \subset V} P(H \times S; p, \textbf{t}) (q-p)^{|V|-|S|}$
\end{theorem}
\begin{proof}
Let $S$ be a set of vertices, $p \leq q$ be positive integers, and consider the set $P_S$ of all $q$-colorings of $H$ that only assign primary colors to vertices of $S$. Let $P(S) = \sum_{f \in P_S} \prod_{e \in P(f)} t_e$. Given such a coloring $f$, $f|_S$ is a $p$-coloring of $H \times S$, and the monochromatic edges of this coloring are the primary edges of $H$.
Moreover, there are $(q-p)^{|V|-|S|}$ colorings $f'$ that only assign primary colors to $S$, and satisfy $f'|_S = f|_S$. 
Thus we see that $P(S) = P(H \times S; p, \textbf{t}) (q-p)^{|V|-|S|}$. Clearly $P(H; p,q, \textbf{t}) = \sum_{S \subset V} P(S)$, and the result follows.
\end{proof}

\subsection{Multivariate Section Polynomial of a Hypergraph}

Recall that we defined the multivariate Section Polynomial by 
$S(q, \textbf{t}) = \sum_{S \subset V} q^{|V| - |S|} \prod_{e \in E(H \times S)} t_{e}$.
Then the next result may be viewed as a multivariate generalization of the result of Dohmen et al \cite{dohmen-ponitz-tittman}.
\begin{proposition}
 $S(H; q, \textbf{t}) = P(H; 1, q+1, \textbf{t})$.
\end{proposition}
\begin{proof}
 We know that $P(H; 1, q+1, \textbf{t}) = \sum_{S \subset V} q^{|V|-|S|} P(H \times S; 1, \textbf{t})$, and clearly $P(H; 1, \textbf{t}) = \prod_{e \in E(H)} t_{e}$ for any hypergraph $H$.
\end{proof}
Note that we obtain the transversal polynomial upon setting $\textbf{t} = 0$.

\subsection{A M\"obius Function Interpretation and Connected Partitions}

Given a partition $\pi$ of $V$, we say that $\pi$ is a connected partition if $H \times S$ is a connected hypergraph for each block $S$ of $\pi$ that is not a singleton. 
Let $\Pi_H$ denote the collection of connected partitions of $H$, ordered by refinement. Observe that $\Pi_H$ is an example of a lattice. It has a unique minimum element, corresponding to 
the partition of $V$ into singletons. It also has a maximum element, corresponding to partitioning $V$ into the vertex sets of the components of $H$.

Given a connected partition $\pi$, and positive integers $p \leq q$, let $f(\pi)$ denote the number of colorings such that:
\begin{itemize}
 \item  $f(u) > p$ if and only if $u$ is a singleton in $\pi$
\item if $e$ is a primary edge, then $e \subset S$ for some block $S$ of $\pi$
\end{itemize}
Note that this definition is equivalent to the definition of $f(\pi)$ given in the paper of Dohmen et al \cite{dohmen-ponitz-tittman}. 
By abuse of notation, we write $e \subset \pi$ to mean the vertex set of $e$ is a subset of some block of $\pi$.
Then $P(H; p, q, \textbf{t}) = \sum_{\pi \in \Pi_H} f(\pi) \prod_{e: e \subset \pi} t_{e}$.
Also, let $k_1(\pi)$ denote the number of singletons of $\pi$. Then $$q^{k_1(\pi)} p^{|\pi| - k_1(\pi)} = \sum_{\sigma \geq \pi} f(\sigma)$$
Using M\"obius inversion and combining these two facts, we obtain:
\begin{theorem}
 $P(H; p, q, \textbf{t}) = \sum_{\pi \leq \sigma \in \Pi_H} q^{k_1(\sigma)} p^{|\sigma|-k_1(\sigma)} \prod_{e: e \subset \pi} t_{e}$.
\end{theorem}
This gives a proof that $P(H; p, q, \textbf{t})$ is a polynomial. This theorem generalizes known results in the 
case $t_e = 0$ \cite{dohmen-ponitz-tittman}, and the case $p = q$ \cite{brylawski-oxley}.

\subsection{Hyperedge Elimination for the Multivariate Chromatic Polynomial}

\begin{theorem}
  $P(H; p, q, \textbf{t})$ satisfies the following:
\begin{enumerate}
 \item $P(E_0; p, q, \textbf{t}) = 1$
\item $P(E_1; p, q, \textbf{t}) = q$
\item $P(H_1 \sqcup H_2; p, q, \textbf{t}) = P(H_1; p,q, \textbf{t}) \cdot P(H_2; p,q,\textbf{t})$
\item for any $e \in E(H)$, we have 
\begin{displaymath}
 \begin{array}{ccl}
  P(H; p,q,\textbf{t}) & = & P(H-e; p,q,\textbf{t}_{\neq e}) \\ & + & (t_{e}-1)P(H/e; p, q, \textbf{t}_{\neq e}) \\ & + & (1-t_{e})(q-p)P(H\dagger e; p, q, \textbf{t}_{\bot e})
 \end{array}
\end{displaymath}
where $\textbf{t}_{\neq e} = \{t_f: f \in E - e \}$, $\textbf{t}_{\bot e} = \{t_f: f \in E \dagger e \}$.
\end{enumerate}
\end{theorem}
\begin{proof}
 The only difficult statement is the hyperedge elimination recurrence. Given a coloring $f$, let $p(H,f) = \prod_{e \in P(f)} t_{e}$. We consider three different types of colorings: if the edge $e$ is primary under the coloring $f$, then 
we say $f$ is of type 3. If $e$ is monochromatic but not primary, we say it is of type $2$. Otherwise, we say $f$ is of type $1$. Let $P_i(H; x,y,z, \textbf{t})$ be the summation previously defined for $P(H; p,q,\textbf{t})$, 
but restricted only to colorings of type $i$. Then $P(H; p,q, \textbf{t}) = P_1(H; p,q, \textbf{t}) + P_2(H; p,q, \textbf{t}) + P_3(H; p,q, \textbf{t})$.
We also divide colorings of $H/e$ into two types. Such a coloring is of type $a$ if $f(v_e) \leq p$, and is of type $b$ otherwise.

Suppose $f$ is of type $f$ is of type $1$. Then $f$ is a coloring of $H-e$ and $P(H-e, f) = p(H, f)$. If $f$ is of type $3$, then we see that $f$ is a coloring of $H - e$ and $f$ induces a coloring on 
$H/e$ by making $v_e$ the same color as the vertices of $e$. Then $p(H, f) = p(H-e, f) + (t_{e} - 1) p(H/e, f')$. Note that $f'$ is a coloring of type $a$.
Finally, suppose $f$ is of type $2$, and let $f'$ be the induced coloring on $H/e$, $f''$ be the coloring restricted to $V-e$. Then $p(H, f) = p(H-e, f) + (t_{e} - 1) p(H/e, f) + (1-t_e)P(H \dagger e, f'')$. 

We see that if we substitute these relations in the summation $P_1 + P_2 + P_3$, we have the term $P(H-e; p, q, \textbf{t})$. Also note that given a coloring $f''$ of $H \dagger e$, and an element $i > p$, we 
obtain a coloring $f$ of type $3$ from $f''$ by assigning all vertices of $e$ the color $i$. We see that we obtain all colorings of type $3$ this way, and so we have obtained $(1-t_e)(q-p)P(H \dagger e; p,q, \textbf{t})$.

The only remaining terms involve colorings of $H/e$, and we obtain $(t_e - 1)P(H/e; p, q, \textbf{t})$.
\end{proof}
Note that this allows us to conclude that $P(H; p,q, \bar{t})$ is a substitutions of $\xi(H; x,y,z, \bar{t})$. It is not hard to show that the converse is also true, so these polynomials are equivalent.

\subsection{Multivariate Chromatic Polynomial for Special Classes of Hypergraphs}

Here we study the multivariate chromatic polynomials of complete $r$-uniform hypergraphs, complete $r$-uniform hyperstars, and sunflower hypergraphs. In some sense, this demonstrates some of the beauty of studying multivariate chromatic polynomials: trying to obtain the equivalent expressions from the hyperedge elimination recurrence, or from subset expansion seem unlikely, but the coloring interpretation makes it much simpler.

We say that an $r$-uniform hypergraph is a hyperstar if $\cap_{e \in H} e \neq \emptyset$. Given $v \in V$, $H$ is a complete $r$-uniform hyperstar centered at $v$ if the edge set consists of all $r$-subsets of $V$ containing $v$ (with no parallel edges). 
Such hypergraphs are unique up to isomorphism, so we define the complete $r$-uniform hyperstar on $[n]$ to consist of all $r$-subsets of $[n]$ containing the vertex $n$. We denote this graph by $H_{n,r}$.

$$P(H_{n,r}; p, q, \textbf{t}) = q^n + \sum_{S: S \subset [n-1], |S| \geq r-1} p (t_{S + n} - 1) (q-1)^{n-|S|-1}$$
where $t_{S+n} = \prod_{T: n \in T \subset S + n, |T| = r} t_T$.
Fix a coloring $f$. Then there exists a set of vertices $S$ such that $f(v) = f(n)$ for all $v \in S$. 
If $|S| > r-1$, and $f(n) \in [p]$, then there are primary edges contained in $S$. Thus we get a contribution of $t_{S+n}$ from this coloring. There are $(x-1)^{n - |S|}$ ways we could have colored the remaining vertices. So there are $p(q-1)^{n-|S|-1}$ colorings $f$ for which $S$ is the set of vertices with the same color as $n$, and $n$ receives a primary color. For each of these, we have a contribution $t_{S+n}$. We also add a term $q^n$ to count all colorings, including those which have no primary edges. Finally, the $-1$ portion of the term $t_{S+n} - 1$ comes from the fact that $q^n$ counts colorings with primary edges, but with the contribution $1$ instead of $t_{S+n}$, so we cancel these colorings out with the minus one term.

Setting $t_e = t$ for all edges $e$, we obtain:
$$P(H_{n,r}; p, q, t) = q^n + \sum_{k=r-1}^{n-1} \binom{n-1}{k} p (t^{\binom{n-1}{k}} - 1) (q-1)^{n-k-1}$$

Another interesting class of hypergraphs is the class of sunflowers. A sunflower hypergraph $H$ is a hypergraph with a set $S \subseteq V$ such that $S \subseteq e$ for every hyperedge $e$, and $\{e \setminus S: e \in E(H) \}$ is a collection of pairwise disjoint sets. Note that we require $H$ to have no parallel hyperedges. We refer to the vertices of $S$ as \emph{seeds}. Let $H$ be a sunflower graph with edge set $\{e_1, \ldots, e_{\ell} \}$.

We have $$P(H; p, q, \textbf{t}) = q^n + \sum_{S \subset E(H)} p(t_S - 1) \prod_{e \not\in S} (q^{|e|-s}-1)$$
where $s$ is the number of seeds of $H$, and $t_S = \prod_{e \in S} t_e$. Given a coloring $f$, if $H$ has any primary edges, then they are all of the same primary color (since every edge contains all 
seeds). Also, since edges are disjoint apart from the set of seeds, we can consider the edges separately. That is, we can choose to consider functions based on the set of edges they leave monochromatic. 
Let $S$ be a set of edges. Then there are $p$ ways to choose a color to make all the edges of $S$ primary, and for every edge not in $S$, there are $q^{|e|-s}$ ways to color the vertices of those edges that are 
not seeds. Thus we obtain the summation. The term $t_S - 1$ comes from the fact that we are enumerating colorings that leave $S$ primary twice: once corresponding to the term $q^n$, which has the wrong weight. 
Thus $t_S - 1$ corrects this weight.

Suppose $S_{r,\ell, s}$ is the $r$-uniform sunflower hypergraph with $\ell$ edges and $s$ seeds. Then we have
$$P(S_{r, \ell, s}; p, q, t) = q^n + \sum_{k=1}^{\ell} p\binom{\ell}{k}(t^k-1)(q^{r-s}-1)^{\ell - k}.$$

\section{Future Work}

We have chosen not to investigate computation complexity questions in this paper. One question is to write an algorithm to determine the coefficients of $\xi(H; x,y,z, \textbf{t})$ or 
$P(H; p, q, \textbf{t})$. For graphs, polynomial time algorithms exist, provided we assume the graphs have bounded tree-width. It seems techniques in this area should work for hypergraphs. 
We ask the following question: given integers $m$, $p$ and $k$, is there a polynomial time algorithm for computing the multivariate chromatic polynomial of a hypergraph $H$, provided the maximum 
size of an edge of $H$ is at most $m$, there are at most $p$ pairwise parallel edges, and $H$ has hypertree-width at most $k$? Is the runtime of such an algorithm a polynomial in $n$, $m$, $p$ and $k$?
Note that the purpose behind including parameters $p$ and $d$ is to ensure that there is a polynomial bound on the number of edges of $H$, and hence the number of variables $t_e$.

Another common question is to determine the computational complexity of evaluating the polynomials in general at a given point. 
However, this question does not make sense for multivariate polynomials, so we have to consider 'labeled' versions of our polynomials, such as what is studied by Averbouch, 
Godlin and Makowsky \cite{averbouch-godlin-makowsky}. We note that one can take such an approach. Rather than define labeled variants of the hyperedge elimination polynomial, 
we state the question only for the hyperedge elimination polynomial. Given a point $(x_0, y_0, z_0)$, and an integer $d$, what is the complexity of determining $\xi(H; x_0, y_0, z_0)$ for 
all $d$-uniform hypergraphs $H$? We restrict the question to $d$-uniform hypergraphs, because otherwise we already know that $\xi(H; x_0, y_0, z_0)$ is at least $\#P$-hard to evaluate for all 
but a finite choice of $(x_0, y_0, z_0)$, because this is already true for graphs. However, for $d > 2$, this does not immediately follow from the case of graphs. 
The proof techniques from graph polynomials seem like they are still applicable.

Finally, there is a notion of mixed-hypergraph coloring. A mixed hypergraph has two types of hyperedges, called type $C$ and type $D$. For a mixed hypergraph, a coloring is proper if no edge 
of type $C$ is monochromatic, and no edge of type $D$ is rainbow. Recall that an edge is rainbow if each of its vertices get distinct colors. There is a chromatic polynomial for mixed 
hypergraphs, and it can be computed usings a recursive algorithm, known as splitting-contraction. One could naturally consider a mixed hypergraph analogue of the multivariate chromatic 
polynomial. It would be interesting to see if there is any deletion-contraction-extraction analogue of splitting-contraction in this case, and if there are interesting evaluations for such 
polynomials.

\bibliographystyle{amsalpha}
\bibliography{graphpolys.bib}

\end{document}